\documentclass[12pt,letterpaper,final,twoside,leqno]{amsart}
\usepackage{amsmath,amssymb,amsthm,amsfonts,amscd,amsopn}
\usepackage{hyperref}
\usepackage[usenames]{color}
\usepackage{eucal, mathrsfs}
\usepackage{xspace}
\usepackage[all]{xy}\xyoption{dvips}
\usepackage{comment} 
\usepackage{paralist}
\usepackage{stmaryrd}
\usepackage{enumitem}
\usepackage{supertabular}
\newcounter{are-there-sections}
\def\me{S\'ANDOR J KOV\'ACS\xspace}

\def\mythanks{Supported in part by NSF Grant %s DMS-0554697 and 
 DMS-0856185, and the
Craig McKibben and Sarah Merner Endowed Professorship in Mathematics at the
University of Washington.}

\def\myaddress{University of Washington, Department of Mathematics, Box 354350,
Seattle, WA 98195-4350, USA}

\def\myemail{skovacs@uw.edu\xspace}

\def\myurladdr{http://www.math.washington.edu/$\sim$kovacs\xspace}

\DeclareMathAlphabet{\smallchanc}{OT1}{pzc}%
                                 {m}{it}
\DeclareFontFamily{OT1}{pzc}{}
\DeclareFontShape{OT1}{pzc}{m}{it}%
             {<-> s * [1.100] pzcmi7t}{}
\DeclareMathAlphabet{\mathchanc}{OT1}{pzc}%
                                 {m}{it}
\newcommand{\mcR}{\mathchanc{R}}

\DeclareFontFamily{OMS}{rsfs}{\skewchar\font'60}
\DeclareFontShape{OMS}{rsfs}{m}{n}{<-5>rsfs5 <5-7>rsfs7 <7->rsfs10 }{}
\DeclareSymbolFont{rsfs}{OMS}{rsfs}{m}{n}
\DeclareSymbolFontAlphabet{\scr}{rsfs}

\newcommand{\sF}{\scr{F}}

\newcommand{\sO}{\scr{O}}

\newcommand{\bC}{\mathbb{C}}

\newcommand{\bH}{\mathbb{H}}

\newcommand{\bP}{\mathbb{P}}

\newcommand{\bZ}{\mathbb{Z}}

\newcommand{\frd}{\mathfrak{d}}

\newcommand{\fpi}{{\sf f^{p,i}}}
\newcommand{\foi}{{\sf f^{0,i}}}
\newcommand{\hi}{{\sf h^{i}}}

\newcommand{\onto}{\twoheadrightarrow}
\newcommand{\isom}{\overset{\simeq\ }\longrightarrow}

\newcommand{\leteq}{\colon\!\!\!=}
\newcommand{\col}{\colon}
\DeclareMathOperator{\Ob}{{Ob}}

\DeclareMathOperator{\supp}{{supp}}

\newcommand{\factor}[2]{\left. \raise 2pt\hbox{\ensuremath{#1}} \right/
        \hskip -2pt\raise -2pt\hbox{\ensuremath{#2}}}
\newcommand{\myR}{{\mcR\!}}

\newcommand{\tld}{\widetilde }
\newcommand{\blank}{\underline{\hskip 10pt}}

\newcommand{\kdot}{{{\,\begin{picture}(1,1)(-1,-2)\circle*{2}\end{picture}\ }}}
\newcommand{\mydot}{\kdot}

\newcommand{\cx}{\sf}
\newcommand{\DuBois}[1]{{\underline \Omega {}^0_{#1}}}
\newcommand{\FullDuBois}[1]{{\underline \Omega {}^{\mydot}_{#1}}}
\newcommand{\Om}{\underline{\Omega}}

\def\coh#1.#2.#3.{H^{#1}(#2,#3)}
\def\dimcoh#1.#2.#3.{h^{#1}(#2,#3)}
\def\hypcoh#1.#2.#3.{\mathbb H_{\vphantom{l}}^{#1}(#2,#3)}
\def\loccoh#1.#2.#3.#4.{H^{#1}_{#2}(#3,#4)}
\def\dimloccoh#1.#2.#3.#4.{h^{#1}_{#2}(#3,#4)}
\def\lochypcoh#1.#2.#3.#4.{\mathbb H^{#1}_{#2}(#3,#4)}
\def\ses#1.#2.#3.{0  \longrightarrow  #1   \longrightarrow 
 #2 \longrightarrow #3 \longrightarrow 0} 
\def\sesshort#1.#2.#3.{0
 \rightarrow #1 \rightarrow #2 \rightarrow #3 \rightarrow 0}
\def\dist#1.#2.#3.{  #1   \longrightarrow 
 #2 \longrightarrow #3 \stackrel{+1}{\longrightarrow} } % \tag{$\bigtriangleup$}}
\def\CDdist#1.#2.#3.{  #1   @>>>  #2  @>>>   #3 @>+1>> }  
\def\shortses#1.#2.#3.{0  \rightarrow  #1   \rightarrow 
 #2  \rightarrow   #3 \rightarrow  0}
\def\shortdist#1.#2.#3.{  #1   \rightarrow 
 #2  \rightarrow   #3 \stackrel{+1}{\rightarrow} }  % \tag{$\bigtriangleup$}}
\def\ddist#1.#2.#3.#4.#5.#6.{\CD
#1 @>>> #2 @>>> #3 @>+1>> \\
@VVV @VVV @VVV \\
#4 @>>> #5 @>>> #6 @>+1>> 
\endCD}
\def\ddistun#1.#2.#3.#4.#5.#6.{\CD
#1 @>>> #2 @>>> #3 @>+1>> \\
@. @VVV @VVV  \\
#4 @>>> #5 @>>> #6 @>+1>> 
\endCD}
\def\Iff#1#2#3{
\hfil\hbox{\hsize =#1
\vtop{\noin #2}
\hskip.5cm 
\lower.5\baselineskip\hbox{$\Leftrightarrow$}\hskip.5cm
\vtop{\noin #3}}\hfil\medskip}
\newcommand{\union}\cup
\newcommand{\intersect}\cap
\newcommand{\Union}\bigcup
\newcommand{\Intersect}\bigcap
\def\myoplus#1.#2.{\underset #1 \to {\overset #2 \to \oplus}}

\newcommand{\resto}{\big\vert_}

\def\qis{\,{\simeq}_{\text{qis}}\,}

\begin{document}
\makeatletter
\newenvironment{refmr}{}{}
\renewcommand{\labelenumi}{{\rm (\thethm.\arabic{enumi})}}
%\renewcommand{\labelenumi}{\hskip .5em(\thethm.\arabic{enumi})}
%%%%%%%%%%%% 
%
\setitemize[1]{leftmargin=*,parsep=0em,itemsep=0.125em,topsep=0.125em}
\newcommand\james{M\raise .575ex \hbox{\text{c}}Kernan}

\renewcommand\thesubsection{\thesection.\Alph{subsection}}
\renewcommand\subsection{
  \renewcommand{\sfdefault}{pag}
  \@startsection{subsection}%
  {2}{0pt}{-\baselineskip}{.2\baselineskip}{\raggedright
    \sffamily\itshape\small
  }}
\renewcommand\section{
  \renewcommand{\sfdefault}{phv}
  \@startsection{section} %
  {1}{0pt}{\baselineskip}{.2\baselineskip}{\centering
    \sffamily
    \scshape
    %\bfseries
}}
%\renewcommand\@cite[2]{{\rm [{#1\ifthenelse{\boolean{@tempswa}}{,\nolinebreak[3] #2}{}}]}}
%%%%%%%%%%
\newcounter{lastyear}\setcounter{lastyear}{\the\year}
\addtocounter{lastyear}{-1}
%%%%%%%%%%%
\newcommand\sideremark[1]{%
\normalmarginpar
\marginpar
[
\hskip .45in
\begin{minipage}{.75in}
\tiny #1
\end{minipage}
]
{
\hskip -.075in
\begin{minipage}{.75in}
\tiny #1
\end{minipage}
}}
\newcommand\rsideremark[1]{
\reversemarginpar
\marginpar
[
\hskip .45in
\begin{minipage}{.75in}
\tiny #1
\end{minipage}
]
{
\hskip -.075in
\begin{minipage}{.75in}
\tiny #1
\end{minipage}
}}
%%%%%%%%%
\newcommand\Index[1]{{#1}\index{#1}}
\newcommand\inddef[1]{\emph{#1}\index{#1}}
\newcommand\noin{\noindent}
\newcommand\hugeskip{\bigskip\bigskip\bigskip}
\newcommand\smc{\sc}
\newcommand\dsize{\displaystyle}
\newcommand\sh{\subheading}
\newcommand\nl{\newline}
%%%%%%%%%%%
%%%%%%%%% bibliography helpers
%%%%%%%%%%%
\newcommand\get{\input /home/kovacs/tex/latex/} %$ 
\newcommand\Get{\Input /home/kovacs/tex/latex/} %$ 
\newcommand\toappear{\rm (to appear)}
\newcommand\mycite[1]{[#1]}
\newcommand\myref[1]{(\ref{#1})}
\newcommand\myli{\hfill\newline\smallskip\noindent{$\bullet$}\quad}
\newcommand\vol[1]{{\bf #1}\ } 
\newcommand\yr[1]{\rm (#1)\ } 
%%%%%%%%%%%
%%%%%%%%% text abbreviations
%%%%%%%%%%%
\newcommand\cf{cf.\ \cite}
\newcommand\mycf{cf.\ \mycite}
\newcommand\te{there exist}
\newcommand\st{such that}
%%%%%%%%%%%
%%%%%%%%%%% Theorem Style: BOZONT
%%%%%%%%%%%
\newcommand\myskip{3pt}
\newtheoremstyle{bozont}{3pt}{3pt}%
     {\itshape}%         Body font
     {}%         Indent amount (empty = no indent, \parindent = para indent)
     {\bfseries}% Thm head font
     {.}%        Punctuation after thm head
     {.5em}%     Space after thm head (\newline = linebreak)
     {\thmname{#1}\thmnumber{ #2}\thmnote{ \rm #3}}%         Thm head spec
%%%%%%%%%%%%%%%%%%%%%%%%%%%%%%
\newtheoremstyle{bozont-sf}{3pt}{3pt}%
     {\itshape}%         Body font
     {}%         Indent amount (empty = no indent, \parindent = para indent)
     {\sffamily}% Thm head font
     {.}%        Punctuation after thm head
     {.5em}%     Space after thm head (\newline = linebreak)
     {\thmname{#1}\thmnumber{ #2}\thmnote{ \rm #3}}%         Thm head spec
%%%%%%%%%%%%%%%%%%%%%%%%%%%%%%
\newtheoremstyle{bozont-sc}{3pt}{3pt}%
     {\itshape}%         Body font
     {}%         Indent amount (empty = no indent, \parindent = para indent)
     {\scshape}% Thm head font
     {.}%        Punctuation after thm head
     {.5em}%     Space after thm head (\newline = linebreak)
     {\thmname{#1}\thmnumber{ #2}\thmnote{ \rm #3}}%         Thm head spec
%%%%%%%%%%%%%%%%%%%%%%%%%%%%%%
\newtheoremstyle{bozont-remark}{3pt}{3pt}%
     {}%         Body font
     {}%         Indent amount (empty = no indent, \parindent = para indent)
     {\scshape}% Thm head font
     {.}%        Punctuation after thm head
     {.5em}%     Space after thm head (\newline = linebreak)
     {\thmname{#1}\thmnumber{ #2}\thmnote{ \rm #3}}%         Thm head spec
%%%%%%%%%%%%%%%%%%%%%%%%%%%%%%
\newtheoremstyle{bozont-def}{3pt}{3pt}%
     {}%         Body font
     {}%         Indent amount (empty = no indent, \parindent = para indent)
     {\bfseries}% Thm head font
     {.}%        Punctuation after thm head
     {.5em}%     Space after thm head (\newline = linebreak)
     {\thmname{#1}\thmnumber{ #2}\thmnote{ \rm #3}}%         Thm head spec
%%%%%%%%%%%%%%%%%%%%%%%%%%%%%%
\newtheoremstyle{bozont-reverse}{3pt}{3pt}%
     {\itshape}%         Body font
     {}%         Indent amount (empty = no indent, \parindent = para indent)
     {\bfseries}% Thm head font
     {.}%        Punctuation after thm head
     {.5em}%     Space after thm head (\newline = linebreak)
     {\thmnumber{#2.}\thmname{ #1}\thmnote{ \rm #3}}%         Thm head spec
%%%%%%%%%%%%%%%%%%%%%%%%%%%%%%
\newtheoremstyle{bozont-reverse-sc}{3pt}{3pt}%
     {\itshape}%         Body font
     {}%         Indent amount (empty = no indent, \parindent = para indent)
     {\scshape}% Thm head font
     {.}%        Punctuation after thm head
     {.5em}%     Space after thm head (\newline = linebreak)
     {\thmnumber{#2.}\thmname{ #1}\thmnote{ \rm #3}}%         Thm head spec
%%%%%%%%%%%%%%%%%%%%%%%%%%%%%%
\newtheoremstyle{bozont-reverse-sf}{3pt}{3pt}%
     {\itshape}%         Body font
     {}%         Indent amount (empty = no indent, \parindent = para indent)
     {\sffamily}% Thm head font
     {.}%        Punctuation after thm head
     {.5em}%     Space after thm head (\newline = linebreak)
     {\thmnumber{#2.}\thmname{ #1}\thmnote{ \rm #3}}%         Thm head spec
%%%%%%%%%%%%%%%%%%%%%%%%%%%%%%
\newtheoremstyle{bozont-remark-reverse}{3pt}{3pt}%
     {}%         Body font
     {}%         Indent amount (empty = no indent, \parindent = para indent)
     {\sc}% Thm head font
     {.}%        Punctuation after thm head
     {.5em}%     Space after thm head (\newline = linebreak)
     {\thmnumber{#2.}\thmname{ #1}\thmnote{ \rm #3}}%         Thm head spec
%%%%%%%%%%%%%%%%%%%%%%%%%%%%%%
\newtheoremstyle{bozont-def-reverse}{3pt}{3pt}%
     {}%         Body font
     {}%         Indent amount (empty = no indent, \parindent = para indent)
     {\bfseries}% Thm head font
     {.}%        Punctuation after thm head
     {.5em}%     Space after thm head (\newline = linebreak)
     {\thmnumber{#2.}\thmname{ #1}\thmnote{ \rm #3}}%         Thm head spec
%%%%%%%%%%%%%%%%%%%%%%%%%%%%%%
\newtheoremstyle{bozont-def-newnum-reverse}{3pt}{3pt}%
     {}%         Body font
     {}%         Indent amount (empty = no indent, \parindent = para indent)
     {\bfseries}% Thm head font
     {}%        Punctuation after thm head
     {.5em}%     Space after thm head (\newline = linebreak)
     {\thmnumber{#2.}\thmname{ #1}\thmnote{ \rm #3}}%         Thm head spec
%%%%%%%%%%%%%%%%%%%%%%%%%%%%%%
%%%%%%%%%%%%%%%%%%%%%%%%%%%%%% thms
%%%%%%%%%%%%%%%%%%%%%%%%%%%%%%
\theoremstyle{bozont}    
\ifnum \value{are-there-sections}=0 {%
  \newtheorem{proclaim}{Theorem}
} 
\else {%
  \newtheorem{proclaim}{Theorem}[section]
} 
\fi
%%%%%%%%%%%%%%%%%%%%%%%%%%%%%%
%%%%%%%%%%%%%%%%%%%%%%%%%%%%%%
\newtheorem{thm}[proclaim]{Theorem}
\newtheorem{mainthm}[proclaim]{Main Theorem}
\newtheorem{cor}[proclaim]{Corollary} 
\newtheorem{cors}[proclaim]{Corollaries} 
\newtheorem{lem}[proclaim]{Lemma} 
\newtheorem{prop}[proclaim]{Proposition} 
\newtheorem{conj}[proclaim]{Conjecture}
\newtheorem{subproclaim}[equation]{Theorem}
\newtheorem{subthm}[equation]{Theorem}
\newtheorem{subcor}[equation]{Corollary} 
\newtheorem{sublem}[equation]{Lemma} 
\newtheorem{subprop}[equation]{Proposition} 
\newtheorem{subconj}[equation]{Conjecture}
%%%%
\theoremstyle{bozont-sc}
\newtheorem{proclaim-special}[proclaim]{\specialthmname}
\newenvironment{proclaimspecial}[1]
     {\def\specialthmname{#1}\begin{proclaim-special}}
     {\end{proclaim-special}}
%%%%%%%%%%%%%%%%%%%%%%%%%%%%%%
\theoremstyle{bozont-remark}
\newtheorem{rem}[proclaim]{Remark}
\newtheorem{subrem}[equation]{Remark}
\newtheorem{notation}[proclaim]{Notation} 
\newtheorem{assume}[proclaim]{Assumptions} 
\newtheorem{obs}[proclaim]{Observation} 
\newtheorem{example}[proclaim]{Example} 
\newtheorem{examples}[proclaim]{Examples} 
\newtheorem{complem}[equation]{Complement}%!!!!!!!!!!!!!!!!!!!!!!
\newtheorem{const}[proclaim]{Construction}   %!!!!!!!!!!!!!!!!
\newtheorem{ex}[proclaim]{Exercise} 
\newtheorem{subnotation}[equation]{Notation} 
\newtheorem{subassume}[equation]{Assumptions} 
\newtheorem{subobs}[equation]{Observation} 
\newtheorem{subexample}[equation]{Example} 
\newtheorem{subex}[equation]{Exercise} 
\newtheorem{claim}[proclaim]{Claim} 
\newtheorem{inclaim}[equation]{Claim} 
\newtheorem{subclaim}[equation]{Claim} 
\newtheorem{case}{Case} 
\newtheorem{subcase}{Subcase}   
\newtheorem{step}{Step}
\newtheorem{approach}{Approach}
\newtheorem{Fact}[proclaim]{Fact}
\newtheorem{fact}{Fact}
\newtheorem{subsay}{}
%%\newcommand{\Subheading}[1]
%%{\def\SubHeadingName{#1}\begin{SubHeading}\end{SubHeading}}%  
\newtheorem*{SubHeading*}{\SubHeadingName}%
\newtheorem{SubHeading}[proclaim]{\SubHeadingName}
\newtheorem{sSubHeading}[equation]{\sSubHeadingName}
\newenvironment{demo}[1] {\def\SubHeadingName{#1}\begin{SubHeading}}
  {\end{SubHeading}}%  
\newenvironment{subdemo}[1]{\def\sSubHeadingName{#1}\begin{sSubHeading}}
  {\end{sSubHeading}} %
\newenvironment{demo-r}[1]{\def\SubHeadingName{#1}\begin{SubHeading-r}}
  {\end{SubHeading-r}}%
\newenvironment{subdemo-r}[1]{\def\sSubHeadingName{#1}\begin{sSubHeading-r}}
  {\end{sSubHeading-r}} %
\newenvironment{demo*}[1]{\def\SubHeadingName{#1}\begin{SubHeading*}}
  {\end{SubHeading*}}%
\newtheorem{defini}[proclaim]{Definition}
\newtheorem{question}[proclaim]{Question}
\newtheorem{subquestion}[equation]{Question}
\newtheorem{crit}[proclaim]{Criterion}
\newtheorem{pitfall}[proclaim]{Pitfall}%!!!!!!!!!!!!!!!!!!!!!!
\newtheorem{addition}[proclaim]{Addition}%!!!!!!!!!!!!!!!!!!!!!!
\newtheorem{principle}[proclaim]{Principle} %!!!!!!!!!!!!!!!!!!!!!!
%%%%%%%%%%%%%%%%%%%%%%%%%%%%%%
%%%  these were at once \theoremstyle{bozont-def}
\newtheorem{condition}[proclaim]{Condition}
\newtheorem{say}[proclaim]{}
\newtheorem{exmp}[proclaim]{Example}
\newtheorem{hint}[proclaim]{Hint}
\newtheorem{exrc}[proclaim]{Exercise}
\newtheorem{prob}[proclaim]{Problem}
\newtheorem{ques}[proclaim]{Question}    %!!!!!!!!!!!!!!!!!!!!
\newtheorem{alg}[proclaim]{Algorithm}
\newtheorem{remk}[proclaim]{Remark}          
\newtheorem{note}[proclaim]{Note}            
\newtheorem{summ}[proclaim]{Summary}         
\newtheorem{notationk}[proclaim]{Notation}   
\newtheorem{warning}[proclaim]{Warning}  
\newtheorem{defn-thm}[proclaim]{Definition--Theorem}  %!!!!!!!!!!!!!!!!!!!!!!!!
\newtheorem{convention}[proclaim]{Convention}  %!!!!!!!!!!!!!!!!!!!!!!!!!!!
%%%%%%%%%%%%%%%%%%%%%%%%%%%%%%
\newtheorem*{ack}{Acknowledgment}
\newtheorem*{acks}{Acknowledgments}
%%%%%%%%%%%%%%%%%%%%%%%%%%%%%%
\theoremstyle{bozont-def}    
\newtheorem{defn}[proclaim]{Definition}
\newtheorem{subdefn}[equation]{Definition}
%%%%%%%%%%%%%%%%%%
\theoremstyle{bozont-reverse}    
\newtheorem{corr}[proclaim]{Corollary} 
\newtheorem{lemr}[proclaim]{Lemma} 
\newtheorem{propr}[proclaim]{Proposition} 
\newtheorem{conjr}[proclaim]{Conjecture}
%%%%
\theoremstyle{bozont-reverse-sc}
\newtheorem{proclaimr-special}[proclaim]{\specialthmname}
\newenvironment{proclaimspecialr}[1]%
{\def\specialthmname{#1}\begin{proclaimr-special}}%
{\end{proclaimr-special}}
%%%%%%%%%%%%%%%%%%%%%%%%%%%%%%
\theoremstyle{bozont-remark-reverse}
\newtheorem{remr}[proclaim]{Remark}
\newtheorem{subremr}[equation]{Remark}
\newtheorem{notationr}[proclaim]{Notation} 
\newtheorem{assumer}[proclaim]{Assumptions} 
\newtheorem{obsr}[proclaim]{Observation} 
\newtheorem{exampler}[proclaim]{Example} 
\newtheorem{exr}[proclaim]{Exercise} 
\newtheorem{claimr}[proclaim]{Claim} 
\newtheorem{inclaimr}[equation]{Claim} 
\newtheorem{SubHeading-r}[proclaim]{\SubHeadingName}
\newtheorem{sSubHeading-r}[equation]{\sSubHeadingName}
\newtheorem{SubHeadingr}[proclaim]{\SubHeadingName}
\newtheorem{sSubHeadingr}[equation]{\sSubHeadingName}
\newenvironment{demor}[1]{\def\SubHeadingName{#1}\begin{SubHeadingr}}{\end{SubHeadingr}}
\newtheorem{definir}[proclaim]{Definition}
%%%%%%%%%%%%%%%%%%%%%%%%%%%%%%
\theoremstyle{bozont-def-newnum-reverse}    
\newtheorem{newnumr}[proclaim]{}
%%%%%%%%%%%%%%%%%%%%%%%%%%%%%%
\theoremstyle{bozont-def-reverse}    
\newtheorem{defnr}[proclaim]{Definition}
\newtheorem{questionr}[proclaim]{Question}
\newtheorem{newnumspecial}[proclaim]{\specialnewnumname}
\newenvironment{newnum}[1]{\def\specialnewnumname{#1}\begin{newnumspecial}}{\end{newnumspecial}}
%%%%%%%%%%%%%%%%%%%%%%%%%%%%%%%%%%%%%%%%%%%%%%%%%%%%%%%%%%%%%%%%%%
%%%%%%%%%%%%%%%%  Labels and Refs for Items %%%%%%%%%%%%%%%%%%%%%%
\newcounter{thisthm} 
\newcounter{thissection} 
\newcommand{\ilabel}[1]{%
  \newcounter{#1}%
  \setcounter{thissection}{\value{section}}%
  \setcounter{thisthm}{\value{proclaim}}%
  \label{#1}}
\newcommand{\iref}[1]{%
  (\the\value{thissection}.\the\value{thisthm}.\ref{#1})}
%%%%%%%%%%%%%%%%%%%%%%%%%%%%%%%%%%%%%%%%%%%%%%%%%%%%%%%%%%%%%%%%%%
%%%%%%%%%%%%%%%%%%%%%%%%%%%%%%%%%%%%%%%%%%%%%%%%%%%%%%%%%%%%%%%%%%
%%%%%%%%%%%%%%%%%%
\numberwithin{equation}{proclaim}
\numberwithin{figure}{section} 
\newcommand\equinsect{\numberwithin{equation}{section}}
\newcommand\equinthm{\numberwithin{equation}{proclaim}}
\newcommand\figinthm{\numberwithin{figure}{proclaim}}
\newcommand\figinsect{\numberwithin{figure}{section}}
%% The next environment produces equations that are numbered within the section, not the
%% theorem. It also increases the counter for thm. This is to be used at the
%% beginning of a section to avoid reference numbers containing 0. It may also be
%% used for equations that are not part of a numbered statement. Otherwise equations
%% take the last numbered environment's number and this is not always desirable.
\newenvironment{sequation}{%
\numberwithin{equation}{section}%
\begin{equation}%
}{%
\end{equation}%
\numberwithin{equation}{proclaim}%
\addtocounter{proclaim}{1}%
}%
%%%%%%%%%%%%%%%
\newcommand{\num}{\arabic{section}.\arabic{proclaim}}
%%%%%%%%%%%%%%%%%
\newenvironment{pf}{\smallskip \noindent {\sc Proof. }}{\qed\smallskip}
%%%%%%%%%%%%%%%%%
\newenvironment{enumerate-p}{
  \begin{enumerate}}
  {\setcounter{equation}{\value{enumi}}\end{enumerate}}
\newenvironment{enumerate-cont}{
  \begin{enumerate}
    {\setcounter{enumi}{\value{equation}}}}
  {\setcounter{equation}{\value{enumi}}
  \end{enumerate}}
%%%%%%%%%%%
\let\lenumi\labelenumi
\newcommand{\rmlabels}{\renewcommand{\labelenumi}{\rm \lenumi}}
\newcommand{\rmlabelsoff}{\renewcommand{\labelenumi}{\lenumi}}
%%%%%%%%%%%
\newenvironment{heading}{\begin{center} \sc}{\end{center}}
%%%%%%%%%%%
\newcommand\subheading[1]{\smallskip\noindent{{\bf #1.}\ }}
%%%%%%%%%%%%%%%%%%%%%%%%%%%%%%
\newlength{\swidth}
\setlength{\swidth}{\textwidth}
\addtolength{\swidth}{-,5\parindent}
\newenvironment{narrow}{
  \medskip\noindent\hfill\begin{minipage}{\swidth}}
  {\end{minipage}\medskip}
%%%%%%%%%%%%%%%%%%%%%%%%%%%%%%
\newcommand\nospace{\hskip-.45ex}
\makeatother
%%

%%%%%%%%%%%%%%%%%%%%%%%%%
\title{The intuitive definition of Du~Bois singularities}
\author{\me}
\thanks{\mythanks}
\address{\myaddress}
\email{\myemail}
\urladdr{\myurladdr}
%\keywords
\maketitle
%\tableofcontents{}
%%%%%%%%%%%%%%%%%%%%%%%%%

\centerline{\it To Gerard van der Geer on the occasion of his 60$^{\text{th}}$
  birthday} % 

\section{Introduction}

Let $X$ be a smooth proper variety. Then the Hodge-to-de-Rham (a.k.a.\ Fr\"olicher)
spectral sequence degenerates at $E_1$ and hence the singular cohomology group
$H^i(X, \bC)$ admits a Hodge filtration 
\begin{sequation}
  \label{eq:ndb8}
  H^{i}(X, \bC)= F^0H^{i}(X, \bC)\supseteq
  F^1H^{i}(X, \bC)\supseteq \dots
\end{sequation}%
and in particular there exists
a natural surjective map
\begin{sequation}\label{eq:ndb1}
H^i(X, \bC)\onto Gr^0_FH^{i}(X, \bC)
\end{sequation}%
where 
\begin{sequation}
  \label{eq:ndb2}
  Gr^0_FH^{i}(X, \bC) \simeq H^i(X,
  \sO_X). 
\end{sequation}

Deligne's theory of (mixed) Hodge stuctures implies that even if $X$ is singular,
there still exists a Hodge filtration and (\ref{eq:ndb1}) remains true, but in
general (\ref{eq:ndb2}) fails.

Du~Bois singularities were introduced by Steenbrink to identify the class of
singularities for which (\ref{eq:ndb2}) remains true as well. However, naturally, one
does not define a class of singularities by properties of proper varieties.
Singularities should be defined by local properties and Du~Bois singularities are
indeed defined locally.

It is known that rational singularities are Du~Bois (conjectured by Steenbrink and
proved in \cite{Kovacs99}) and so are log canonical singularities (conjectured by
Koll\'ar and proved in \cite{KK10}). These properties make Du~Bois singularities very
important in higher dimensional geometry, especially in moduli theory (see
\cite{SingBook} for more details on applications).

Unfortunately the definition of Du~Bois singularities is rather technical.  The most
important and useful fact about them is the consequence of \eqref{eq:ndb1} and
\eqref{eq:ndb2} that if $X$ is a proper variety over $\bC$ with Du~Bois
singularities, then the natural map
\begin{sequation}
  \label{eq:ndb3}
  H^i(X, \bC)\onto H^i(X, \sO_X)
\end{sequation}
is surjective.

One could try to take this as a definition, but it would not lead to a good result
for two reasons. As mentioned earlier, singularities should be defined locally and it
is not at all likely that a global cohomological assumption would turn out to be a
local property. Second, this particular condition could obviously hold
``accidentally'' and lead to the inclusion of singular spaces that should not be,
thereby further lowering the chances of having a local description of this class of
singularities.

Therefore the reasonable approach is to keep Steenbrink's original defition, after
all it has been proven to define a useful class. It does satisfy the first
requirement above: it is defined locally. Once that is accepted, one might still
wonder if proper varieties with Du~Bois singularities could be characterized with a
property that is close to requiring that \eqref{eq:ndb3} holds.

The main result of the present paper is exactly a characterization like that. 

As we have already observed, simply requiring that \eqref{eq:ndb3} holds is likely to
lead to a class of singularities that is too large. A more natural requirement is to
ask that \eqref{eq:ndb2} holds. Clearly, \eqref{eq:ndb2} implies \eqref{eq:ndb3} by
\eqref{eq:ndb1}, so our goal requirement is indeed satisfied.

The definition \cite[(3.5)]{Steenbrink83} of Du~Bois singularities easily implies
that if $X$ has Du~Bois singularities and $H\subset X$ is a general member of a
basepoint-free linear system, then $H$ has Du~Bois singularities as well. Therefore
it is reasonable that in trying to give an intuitive definition of Du~Bois
singularities, one may assume that the defining condition holds for the
intersection of general members of a fixed basepoint-free linear system.

I will prove here that this is actually enough to characterize Du~Bois singularities
(see \eqref{def:db} for their definition). This result is not geared for
applications, it is mainly interesting from a philosophical point of view. It says
that the local definition not only achieves the desired property for proper
varieties, but does it in an economical way: it does not allow more than it has to.

At the same time, a benefit of this characterization is the fact that for the
uninitiated reader this provides a relatively simple criterion without the use of
derived categories or resolutions directly. In fact, one can make the condition
numerical. This is a trivial translation of the ``real'' statement, but further
emphasizes the simplicity of the criterion.

In order to do this we need to define some notation: Let $X$ be a proper algebraic
variety over $\bC$ and consider Deligne's Hodge filtration $F^\kdot$ on $H^{i}(X,
\bC)$ as in \eqref{eq:ndb8}. Let
$$
Gr^p_FH^{i}(X,\bC)= \factor{F^pH^i(X,\bC)}{F^{p+1}H^i(X,\bC)}
$$
and 
$$
\fpi(X) = \dim_{\bC} Gr^p_FH^{i}(X,\bC).
$$
I will also use the usual notation
$$
\hi(X,\sO_X) = \dim_{\bC} H^{i}(X,\sO_X).
$$

Recall (cf.~\eqref{obs:natural-map}) that by the construction of the Hodge filtration
and the degeneration of the Hodge-to-de-Rham spectral sequence at $E_1$, the natural
surjective map from $H^i(X,\bC)$ factors through $H^i(X,\sO_X)$:\smallskip
$$
\xymatrix{%
  H^i(X, \bC) \ar[r] \ar@/^{13pt}/@{->>}[rr] & H^i(X, \sO_{X}) \ar[r] &
  Gr^0_FH^i(X,\bC).  }
$$
In particular, the natural morphism
\begin{sequation}
  \label{eq:ndb6}
  H^i(X, \sO_{X}) \onto Gr^0_FH^{i}(X,\bC)
\end{sequation}
is also surjective and hence
\begin{sequation}
  \label{eq:ndb4}
  \hi(X,\sO_X)\geq \foi(X).
\end{sequation}

Now we are ready for the main theorem. It essentially says that if the opposite
inequality of \eqref{eq:ndb4} holds for general complete intersections, then the
ambient variety has Du~Bois singularities.

More precisely I will prove the following.

\begin{thm}
  \label{thm:numerical}
  Let $X$ be a proper variety over $\bC$ with a fixed %(non-empty)
  basepoint-free linear system $\frd$. (For instance, $X$ is projective with a fixed
  projective embedding).  Then $X$ has only Du~Bois singularities if and only if
  $\hi(L,\sO_L)\leq \foi(L)$ for $i>0$ for any $L\subseteq X$ which is the
  intersection of general members of $\frd$.
\end{thm}

\begin{cor}
  \label{cor:isolated}
  Let $X\subseteq \bP^N$ be a projective variety over $\bC$ with only isolated
  singularities.  Then $X$ has only Du~Bois singularities if and only if
  $\hi(X,\sO_X)\leq \foi(X)$ for $i>0$.
\end{cor}

\begin{proof}
  As $X$ has only isolated singularities, a general hyperplane section is smooth and
  does not contain any of the singular points. Hence as soon as $\hi(X,\sO_X)\leq
  \foi(X)$ one also has that $\hi(L,\sO_L)\leq \foi(L)$ for any $L\subseteq X$ which
  is the intersection of general hyperplanes in $\bP^N$. Therefore the statement
  follows from \eqref{thm:numerical}.
\end{proof}

These statements reiterate the fact that singularities impose restrictions on global
cohomological conditions. In particular one has the following ad hoc consequence:

\begin{cor}
  Let $X\subseteq \bP^N$ be a projective variety over $\bC$ with only isolated
  singularities.  Assume that $\hi(X,\sO_X)=0$ for $i>0$. Then $X$ has only Du~Bois
  singularities.
\end{cor}

\begin{proof}
  As $\foi(X)\geq 0$, the statement follows from \eqref{cor:isolated}.
\end{proof}

Observe that \eqref{eq:ndb6} combined with the condition $\hi(L,\sO_L)\leq \foi(L)$
implies that $H^i(L, \sO_L) \to Gr^0_FH^i(L,\bC)$ is an isomorphism and hence
\eqref{thm:numerical} follows from the following.

\begin{thm}\label{thm:naive-db}
  Let $X$ be a proper variety over $\bC$ with a fixed basepoint-free linear system
  $\frd$.  Then $X$ has only Du~Bois singularities if and only if for all $i>0$ and
  for any $L\subseteq X$, which is the intersection of general members of $\frd$, the
  natural map,
  \begin{equation*}
%    \label{eq:ndb7}
    \nu_i=\nu_i(L):H^i(L, \sO_L) \to Gr^0_FH^i(L,\bC)
  \end{equation*}
  given by Deligne's theory \cite{MR0498551,MR0498552,Steenbrink83,GNPP88}
  (cf.~\eqref{obs:natural-map}) is an isomorphism for all $i$.
\end{thm}

\begin{rem}\label{rem:one-direction}
  It is clear that if $X$ has only Du~Bois singularities then $\nu_i(L)$ is an
  isomorphism for all $L$.  Therefore the interesting statement of the theorem is
  that the condition above implies that $X$ has only Du~Bois singularities.
\end{rem}

\begin{demo}{\bf Definitions and Notation}\label{demo:defs-and-not}
  Unless otherwise stated, all objects are assumed to be defined over $\bC$, all
  schemes are assumed to be of finite type over $\bC$ and a morphism means a morphism
  between schemes of finite type over $\bC$.

  Let $X$ be a complex scheme (i.e., a scheme of finite type over $\bC$) of dimension
  n. Let $D_{\rm filt}(X)$ denote the derived category of filtered complexes of
  $\sO_{X}$-modules with differentials of order $\leq 1$ and $D_{\rm filt, coh}(X)$
  the subcategory of $D_{\rm filt}(X)$ of complexes $\cx K$, such that for all $i$,
  the cohomology sheaves of $Gr^{i}_{\rm filt}K^{\kdot}$ are coherent cf.\
  \cite{DuBois81}, \cite{GNPP88}.  Let $D(X)$ and $D_{\rm coh}(X)$ denote the derived
  categories with the same definition except that the complexes are assumed to have
  the trivial filtration.  The superscripts $+, -, b$ carry the usual meaning
  (bounded below, bounded above, bounded).  Isomorphism in these categories is
  denoted by $\qis$.  A sheaf $\sF$ is also considered as a complex $\sF^\kdot$ with
  $\sF^0=\sF$ and $\sF^i=0$ for $i\neq 0$.  If $K^{\kdot}$ is a complex in any of the
  above categories, then $h^i(K^{\kdot})$ denotes the $i$-th cohomology sheaf of
  $K^{\kdot}$.

  The right derived functor of an additive functor $F$, if it exists, is denoted by
  $\myR F$ and $\myR^iF$ is short for $h^i\circ \myR F$. Furthermore $\bH^i$
  %, $\bH^i_{\rm c}$, $\bH^i_Z$ , and $\sH^i_Z$ 
  will denote $\myR^i\Gamma$,
  % $\myR^i\Gamma_{\rm c}$, $\myR^i\Gamma_Z$, and $\myR^i\sH_Z$ respectively, 
  where $\Gamma$ is the functor of global sections

%  , $\Gamma_{\rm c}$ is the functor
%  of global sections with proper support, $\Gamma_Z$ is the functor of global
%  sections with support in the closed subset $Z$, and $\sH_Z$ is the functor of the
%  sheaf of local sections with support in the closed subset $Z$.  Note that according
%  to this terminology, if $\phi\col Y\to X$ is a morphism and $\sF$ is a coherent
%  sheaf on $Y$, then $\myR\phi_*\sF$ is the complex whose cohomology sheaves give
%  rise to the usual higher direct images of $\sF$.

%%  We will often use the notion that a morphism ${f}: \sfA\to \sfB$ in a derived
%%  category \emph{has a left inverse}. This means that there exists a morphism
%%  $f^\ell: \sfB\to \sfA$ in the same derived category such that
%%  $f^\ell\circ{f}:\sfA\to\sfA$ is the identity morphism of $\sfA$. I.e., $f^\ell$
%%  is  a \emph{left inverse} of ${f}$. 

%%  Finally, we 
%  I will use the following simplification in notation. First observe
%  that if $\iota:\Sigma \into X$ is a closed embedding of schemes then $\iota_*$ is
%  exact and hence $\myR\iota_*=\iota_*$. This allows one to make the following
%  harmless abuse of notation: If $\sfA\in\ob D(\Sigma)$, then, as usual for sheaves,
%  we will drop $\iota_*$ from the notation of the object $\iota_*\sfA$. In other
%  words, I will, without further warning, consider $\sfA$ an object in $D(X)$.
\end{demo}

\section{Hyperresolutions and Du~Bois' original definition}

We will start with Du~Bois's generalized De Rham complex. The original construction
of the Deligne-Du~Bois's complex, $\FullDuBois{X}$, is based on simplicial
resolutions. The reader interested in the details is referred to the original article
\cite{DuBois81}.  Note also that a simplified construction was later obtained in
\cite{Carlson85} and \cite{GNPP88} via the general theory of polyhedral and cubic
resolutions.  An easily accessible introduction can be found in \cite{Steenbrink85}.

The word ``hyperresolution'' will refer to either simplicial, polyhedral, or cubic
resolution. Formally, the construction of $\FullDuBois{X}$ is the same regardless the
type of resolution used and no specific aspects of either types will be used.

\begin{thm}[{\cite[6.3, 6.5]{DuBois81}}]\label{defDB}
  Let $X$ be a complex scheme of finite type and $D$ a closed subscheme whose
  complement is dense in $X$. Then there exists a unique object $\FullDuBois X \in
  \Ob D_{\rm filt}(X)$ such that using the notation
  $$
  \Om^p_X\leteq Gr^{p}_{\rm filt}\, \FullDuBois{X}[p],
  $$
  it satisfies the following properties
  \begin{enumerate}
  \item %Let $j\col X\setminus D\to X$ be the inclusion map. Then
    $ \FullDuBois{X} \qis \bC_{X} $, i.e., $\FullDuBois{X}$ is a resolution of the
    constant sheaf $\bC$ on $X$.

  \item $\underline{\Omega}_{(\blank)}^{\mydot}$ is functorial, i.e., if $\phi \col
    Y\to X$ is a morphism of proper complex schemes of finite type, then there exists
    a natural map $\phi^{*}$ of filtered complexes
    $$
    \phi^{*}\col \FullDuBois{X} \to R\phi_{*}\underline{\Omega}_Y^{\mydot}.
    $$
    Furthermore, $\FullDuBois{X} \in \Ob \left(D^{b}_{filt, coh}(X)\right)$ and if
    $\phi$ is proper, then $\phi^{*}$ is a morphism in $D^{b}_{filt, coh}(X)$.
    \ilabel{functorial}

  \item Let $U \subseteq X$ be an open subscheme of $X$. Then
    $$
    \FullDuBois{X}\resto U %\big|_ U
    \qis\underline{\Omega}^{\,\mydot}_U.
    $$

  \item If $X$ is proper, there exists a spectral sequence degenerating at $E_1$ and
    abutting to the singular cohomology of $X$ such that the resulting filtration
    coincides with Deligne's Hodge filtration:
    $$
    E_1^{pq}={\bH}^q \left(X, \Om^p_X\right) \Rightarrow H^{p+q}(X, \bC).
    $$\ilabel{item:Hodge}
    In particular,
    $$Gr^p_FH^{p+q}(X,\bC)\simeq {\bH}^q \left(X, \Om^p_X\right).$$

  \item If\/ $\varepsilon_\mydot\col X_\mydot\to X$ is a hyperresolution, then
    $$
    \FullDuBois{X}\qis \myR{\varepsilon_\mydot}_* \Omega^\mydot_{X_\mydot}.
    $$
    In particular, $h^i\left(\Om^p_X\right)=0$ for $i<0$.

  \item Let $H\subset X$ be a general member of a basepoint-free linear system.
    Then
    $$\Om_{H}^\kdot\qis \Om_{X}^\kdot\otimes_L\sO_H$$
    \ilabel{item:db-cx-of-hyper-rel}

  \item There exists a natural map, $\sO_{X}\to \Om^0_X$, compatible with
    \iref{functorial}. \ilabel{item:dR-to-DB}

  \item If\/ $X$ is smooth, then
    $$
    \FullDuBois{X}\qis\Omega^\mydot_X.
    $$
    In particular,
    $$
    \Om^p_X\qis\Omega^p_X.
    $$ 
    \ilabel{item:testing}

  \item If\/ $\phi\col Y\to X$ is a resolution of singularities, then
    $$
    \Om_X^{\dim X}\qis \myR\phi_*\omega_Y.
    $$

  \item\ilabel{item:exact-triangle} If $\pi : \tld Y \rightarrow Y$ is a projective
    morphism, $X \subset Y$ is a reduced closed subscheme such that $\pi$ is an
    isomorphism outside of $X$, $E$ is the reduced subscheme of $\tld Y$ with support
    equal to $\pi^{-1}(X)$, and $\pi' : E \rightarrow X$ is the induced map, then for
    each $p$ one has an exact triangle in the derived category,
    $$
    \xymatrix{ \Om^p_Y \ar[r] & \Om^p_X \oplus \myR \pi_* \Om^p_{\tld Y} \ar[r]^-{-}
      & \myR
      \pi'_* \Om^p_E \ar[r]^-{+1} & .\\
    }
    $$
  \end{enumerate}
\end{thm}

It turns out that the Deligne-Du~Bois complex behaves very much like the de~Rham
complex for smooth varieties. Observe that \iref{item:Hodge} says that the
Hodge-to-de~Rham spectral sequence works for singular varieties if one uses the
Deligne-Du~Bois complex in place of the de~Rham complex. This has far reaching
consequences and if the associated graded pieces, $\Om^p_X$ turn out to be
computable, then this single property leads to many applications.

\begin{obs}\label{obs:natural-map}
  Notice that \iref{item:dR-to-DB} gives a natural map $\sO_{X}\to \Om^0_X$.  This
  implies that the natural map $H^i(X, \bC) \rightarrow \bH^i(X,
  \DuBois{X})$, which is surjective when $X$ is proper because of the degeneration at
  $E_1$ of the spectral sequence in \iref{item:Hodge}, factors as
  $$
  \xymatrix{%
    H^i(X, \bC) \ar[r] \ar@/^{13pt}/[rr] & H^i(X, \sO_{X}) \ar[r] & \bH^i(X,
    \DuBois{X})= & \hskip-2.35em Gr^0_FH^i(X,\bC).  }
  $$
  The induced map $H^i(X, \sO_{X}) \to Gr^0_FH^i(X,\bC)$ is the one that
  appears in \eqref{thm:naive-db}.
\end{obs}

\begin{defini}\label{def:db}
  A scheme $X$ is said to have \emph{Du~Bois singularities} (or \emph{DB
    singularities} for short) if the natural map $\sO_{X}\to \Om^0_X$ from
  \iref{item:dR-to-DB} is a quasi-isomorphism.
\end{defini}

\begin{rem}
  If $\varepsilon_\mydot : X_{\mydot} \rightarrow X$ is a hyperresolution of $X$ then
  $X$ has Du~Bois singularities if and only if the natural map $\sO_X \rightarrow
  \myR {\varepsilon_{\mydot}}_* \sO_{X_{\mydot}}$ is a quasi-isomorphism.
\end{rem}

\begin{example}
  It is easy to see that smooth points are Du~Bois and Deligne proved that normal
  crossing singularities are Du~Bois as well cf.\ \cite[Lemme 2(b)]{MR0376678}.
\end{example}

\section{The proof of {\bf\sf (\ref{thm:naive-db})}}

\equinsect

As observed in \eqref{rem:one-direction}, we only need to prove that if for every
$i>0$ and for every $L\subseteq X$ which is the intersection of general members of
$\frd$, the natural map
\begin{sequation}\label{eq:hodge-restriction}
  \nu_i:H^i(L, \sO_L) \to Gr^0_FH^i(L,\bC)
\end{sequation}
is an isomorphism, then $X$ has Du~Bois singularities.

\begin{obs}\label{obs:i=0}
  Note that it follows that $\nu_i$ is an isomorphism for all $i\in\bZ$. Indeed, both
  sides are zero for $i<0$ and have the same dimension for $i=0$. Since $\nu_i$ is
  surjective this implies the claim.
\end{obs}

Let $\Sigma_X\subseteq X$ denote the locus of points where $X$ does not have Du~Bois
singularities, i.e., $\Sigma_X$ is the smallest closed subset of $X$ such that
$X\setminus \Sigma_X$ has Du~Bois singularities. We would like to prove that
$\Sigma_X=\emptyset$.

Let $H$ be a general member of $\frd$. Then $\Sigma_H=\Sigma_X\cap H$ by
\iref{item:db-cx-of-hyper-rel}.  As our goal is to prove that $\Sigma_X=\emptyset$,
we may replace $X$ with an intersection of general members of $\frd$ and assume that
$\Sigma_X$ is finite.

Consider the \emph{DB defect} of $X$ \cite[2.9]{Kovacs10a}, that is, the mapping cone
of the natural morphism $\sO_X\to \Om_X^0$. By definition there exists an exact
triangle, 
\begin{equation}
  \label{eq:db-defect}
  \xymatrix{%
    \sO_X \ar[r] & \Om^0_X \ar[r] & \Om^\times_X \ar[r]^-{+1} & ,
  }
\end{equation}
and by \eqref{obs:i=0} and \iref{item:Hodge},
$$
H^i(X,\sO_X)\isom \bH^i(X,\Om_X^0)
$$
is an isomorphism for all $i\in\bZ$. It follows that then
\begin{equation}
  \label{eq:coh-of-db-defect}
  \bH^i(X,\Om_X^\times)=0
\end{equation}
for all $i\in\bZ$.

On the other hand there exists a spectral sequence computing
$\bH^i(X,\Om_X^\times)$:
\begin{equation*}
  H^p(X,h^q(\Om_X^\times)) \Rightarrow \bH^{p+q}(X,\Om_X^\times).
\end{equation*}
Observe that $\supp h^q(\Om_X^\times)\subseteq \Sigma_X$ and hence
$0$-dimensional. Consequently 
$$
H^p(X,h^q(\Om_X^\times))=0
$$
for $p>0$, and hence
$$
\bH^{i}(X,\Om_X^\times)=H^0(X,h^i(\Om_X^\times))=h^i(\Om_X^\times)
$$
for all $i\in\bZ$.  Comparing with (\ref{eq:coh-of-db-defect}) we obtain that
$h^i(\Om_X^\times)=0$ for all $i\in\bZ$ and hence $\Om_X^\times\qis 0$. By the
definition of the DB defect this implies (cf.~(\ref{eq:db-defect})) that $X$ has
Du~Bois singularities.  
This proves \eqref{thm:naive-db} and by \eqref{rem:one-direction} that implies
\eqref{thm:numerical}. \qed

\begin{ack}
  The results in this paper were inspired by many conversations with J\'anos
  Koll\'ar, most recently while we both enjoyed the hospitality of the Research
  Institute for Mathematical Sciences at Kyoto University.

  I would also like to thank Karl Schwede and Zsolt Patakfalvi for insightful
  comments.
\end{ack}

\def\cprime{$'$} \def\polhk#1{\setbox0=\hbox{#1}{\ooalign{\hidewidth
  \lower1.5ex\hbox{`}\hidewidth\crcr\unhbox0}}} \def\cprime{$'$}
  \def\cprime{$'$} \def\cprime{$'$} \def\cprime{$'$}
  \def\polhk#1{\setbox0=\hbox{#1}{\ooalign{\hidewidth
  \lower1.5ex\hbox{`}\hidewidth\crcr\unhbox0}}} \def\cdprime{$''$}
  \def\cprime{$'$} \def\cprime{$'$} \def\cprime{$'$} \def\cprime{$'$}
\providecommand{\bysame}{\leavevmode\hbox to3em{\hrulefill}\thinspace}
\providecommand{\MR}{\relax\ifhmode\unskip\space\fi MR}
\providecommand{\MRhref}[2]{%
  \href{http://www.ams.org/mathscinet-getitem?mr=#1}{#2}
}
\providecommand{\href}[2]{#2}

\end{document}